\documentclass[12pt,a4paper,twoside]{amsart}
\usepackage{amsmath}
\usepackage{amssymb}
\usepackage{amsfonts}
\usepackage{a4}

\newtheorem{theorem}{Theorem}[section]
\newtheorem{lemma}[theorem]{Lemma}
\newtheorem{proposition}[theorem]{Proposition}
\newtheorem{corollary}[theorem]{Corollary}

\theoremstyle{definition}
\newtheorem{definition}[theorem]{Definition}
\newtheorem{claim}[theorem]{Claim}

\newtheorem{conjecture}[theorem]{Conjecture}
\newtheorem{remark}[theorem]{Remark}

\newcommand{\mC}{{\mathbb C}}

\newcommand{\mF}{\mathbb F}

\newcommand{\bV}{\mathbf V}

\newcommand{\mZ}{{\mathbb Z}}

\newcommand{\bY}{{\mathbf Y}}
\newcommand{\bX}{{\mathbf X}}
\newcommand{\bP}{{\mathbf P}}
\newcommand{\bA}{{\mathbf A}}
\newcommand{\bB}{{\mathbf B}}
\newcommand{\bC}{{\mathbf C}}

\newcommand{\bG}{{\mathbf G}}
\newcommand{\bZ}{{\mathbf Z}}
\newcommand{\ho}{\hookrightarrow}

\newcommand{\D}{\Delta}

\newcommand{\kk}{\kappa}

\newcommand{\mcP}{\mathcal P}

\newcommand{\mcS}{\mathcal S}

\newcommand{\ti}{\tilde}

\newcommand{\emp}{\emptyset}

\newcommand{\Hom}{\mathrm{Hom}}
\newcommand{\Aut}{\mathrm{Aut}}

\newcommand{\sm}{\setminus}

\newcommand{\codim}{\operatorname{codim}}
\newcommand{\sms}{\smallskip}
\newcommand{\ms}{\medskip}

\newcommand{\fK}{{\mathbf k}}
\usepackage{xcolor}
\usepackage{amscd}

\begin{document} 
\title[Comparison of  the Schmidt and analytic ranks for trilinear forms] {On the Schmidt and analytic ranks for trilinear forms}

\author{Karim Adiprasito,  David Kazhdan and Tamar Ziegler}

\thanks{K. A. is supported by ERC grant 716424 - CASe and ISF Grant 1050/16. T. Z. is supported by ERC grant ErgComNum 682150 and ISF grant 2112/20. }

\maketitle

\begin{abstract} We discuss relations between different notions of ranks for multilinear forms. In particular we show that   
the Schmidt and the analytic ranks for trilinear forms are essentially proportional.
 \end{abstract}
\section{Introduction}

In recent years there is growing interest in properties of polynomials which are independent of the number of variables. 
To study such properties various notions were introduced for measuring the complexity of polynomials. 
In this paper we compare three notion of rank of polynomials $P$ over a field $\fK$ - the Schmidt rank  $r_\fK(P)$, the slice rank  $s_\fK(P)$ both defined for arbitrary fields $\fK$, and the analytic rank $a_\fK(P)$ defined for finite fields $\fK =\mF _q$.

\begin{definition}\label{11} Let $\fK$ be a field,  $\bV$ be a finite-dimensional k-vector space, $V=\bV (\fK)$, let 
$\mcS _{\fK} = \bigoplus _{d\geq 0}\mcS ^d_{\fK}$ be the graded algebra of polynomial functions  on $\bV$ and let
 $P \in \mcS ^d_\fK$.
\begin{enumerate}
\item The \emph{Schmidt rank} $r_\fK(P)$ is  the minimal number $r$ such that $P$ can be written in the form $P=\sum _{i=1}^rQ_iR_i$, where $Q_i,R_i \in \mcS $ are polynomials on $V$ of degrees $<d$.
\item The \emph{slice rank} $s_\fK(P)$ is the minimal number $r$ such that $P$ can be written in the form $P=\sum _{i=1}^rQ_iR_i$, where $\deg(Q_i)=1$.
\item In the case when  $\fK=\mF_q$ is a finite field and 
$\psi : \fK\to \mC ^\star$  a non-trivial additive character
we write $A_{\fK , \psi}( P):= \frac {\sum _{v\in V}\psi (P(v))}
{q^{ \mathrm{dim}(V)}}$. We define the {\it analytic rank} 
$a_{\fK ,\psi}(P): =- \log _q (| A_{\mF_q, \psi}(P) |)$ 
 \end{enumerate}
\end{definition}

 These three notions of complexity of polynomials play an important role in many problems in number theory, additive combinatorics, and algebraic geometry. 
The Schmidt rank, also called the $h$-invariant, was first introduced by Schmidt in his paper on integer points in varieties defined over the rationals. In his paper Schmidt showed that over the complex field the Schmidt rank of a polynomials is proportional to the codimension of the singular locus of the associated variety.  The same notion of complexity was reintroduced in the work of Ananyan and Hochster \cite{ah} as the {\em strength} of a polynomial and was used in  the proof of the Stillman conjecture. The notion of slice rank played an important role in the arguments  for the cap set problem in combinatorics  \cite{clp, eg, tt}. Finally the notion of analytic rank for polynomials over finite fields was introduced in Gower and Wolf  \cite{gw} as a tool for studying the CS-complexity of systems of linear equations. 

\begin{remark}
 It is clear that $r_\fK(P) \leq s_\fK(P) $ and   $r_\fK(P) = s_\fK(P) $ if $d\leq 3.$
\end{remark}

In this paper we consider only  the case of multilinear polynomials. Namely  $V= V_1\times \dots \times V_d$ is a product of $\fK$-vector spaces and  $P: V_1\times \dots \times V_d \to \fK$ is a multilinear polynomial. 

We introduce two additional definitions.
\begin{definition}\label{d1}
\begin{enumerate}
\item In the case when  $V= V_1\times \dots \times V_d$ is a product of $\fK$-vector spaces and
 $P: V_1\times \dots \times V_d \to \fK$ is a multilinear polynomial we define a $\fK$-subvariety  $\bZ _P \subset \bV _2\times \dots \times \bV _d $ by $\bZ _P :=\{(v_2,\dots ,v_d)\in 
 \bV _2\times \dots \times \bV _d | P(v_1, v_2,\dots ,v_d)=0, \forall v_1\in V_1\}$.
\item For  a subspace $L\subset (  V_1 \otimes \ldots \otimes  V_d)^\vee$ we define 
$r_\fK(L)$ as the minimum of $ \sum_{i=1}^d \mathrm{codim}(W_i)$ over the set of subspaces $W_i\subset V_i$, $i=1, \ldots, m$, such that $l_{| W_1\otimes  \ldots \otimes W_d} \equiv 0$ for all $l\in L$.
 \end{enumerate}
\end{definition}

\begin{claim}\label{3} For a multilinear polynomial 
$P: V_1\times \dots \times V_d \to \fK$ 
the slice rank  $s_{\fK} (P)$ is equal to the minimum of
 $\sum _{i=1}^d \codim _{V_i}(W_i) $ where 
$W_i\subset V_i,1\leq i\leq d$ are subspaces such that 
$P_{|W_1 \times \dots \times W_d  }\equiv 0$. It follows that if $L$ is one dimensional  $L=\fK P$ then $r_{\fK} (L)=s_{\fK} (P)$. 
\end{claim}

\begin{lemma}\label{ind} In the case when $\fK =\mF _q$ we have $A_{\fK , \psi }(P)=\frac{| \bZ _P (\mF_q)|} {q^{\sum _{i=2}^d \mathrm{dim}(V_i)}} $.
\end{lemma}

\begin{proof} For $(v_2,\dots ,v_n)\in V_2\times \ldots \times V_d $ we define $\psi  _{v_2,\dots ,v_n }:V_1\to \mC ^\star$ by 
\[ \psi _{v_2,\dots ,v_n } (v_1):= \psi (P(v_1,v_2,\dots ,v_n)).
\]
 It is clear that $ \psi _{v_2,\dots ,v_n }$ is an additive character of $V_1$ which trivial if and only if $ (v_2,\dots ,v_n)\in  
\bZ _P (\mF_q) $.
 Therefore $\sum _{v_1\in V_1} \psi  _{v_2,\dots ,v_n } (v_1)=0$ if $v_1\not \in \bZ _P (\mF_q) $ and $\sum _{v_1\in V_1} \psi  _{v_2,\dots ,v_n } (v_1)=q^{\mathrm{dim}(V_1)}$ if $v_1\in \bZ _P (\mF_q) $. 
\end{proof}

\begin{corollary} The analytic rank  of a  multilinear polynomial $P$ does not depend on the choice of the non-trivial additive character $\psi$. We will denote it by $a_\fK (P)$.
\end{corollary}
  
Fro a field $\fK$ we denote by $\bar \fK$ the algebraic closure.  Obviously we have $ r_{\bar \fK}(P) \le  r_\fK (P)$.  We conjecture that the reverse inequality is essentially holds as well.  
\begin{conjecture}[d]\label{d} For any $d\geq 2$ there exists $\kk _d>0$ such that 
$ r_\fK (P) \leq \kk _d r_{\bar \fK}(P)$ for any  field $\fK$ and a multilinear $\fK$-polynomial $P$ of degree $d$,  where $\bar \fK$ is the algebraic closure of $\fK$.
\end{conjecture}

\begin{remark}
 It is easy to see that  Conjecture \ref{d} holds for $d=2$  with $\kk _2=1$.
\end{remark}

 In \cite{D} Derksen introduced the notion of the {\em $G$-stable rank} of a multilinear polynomial (denoted $r_{\fK}^G(P)$) defined in terms of Geometric Invariant Theory and proved the following
\begin{enumerate}
\item $r_{\fK}^G(P) = r_{\bar \fK}^G(P) $ and 
\item $\frac {2s _\fK(P)}{d}\leq r_{\fK}^G(P) \leq s _\fK(P)$
\end{enumerate}
This immediately implies 
the inequality $s_\fK (P)\leq \frac{3}{2} s_{\bar \fK}(P)$ for trilinear polynomials $P$.
Since $s_\fK (P)= r_\fK (P)$ for trilinear polynomials $P$ we obtain the following result.

\begin{theorem}[Dersken]\label{33}  For trilinear polynomials we have $r_\fK (P) \leq \frac{3}{2}r_{\bar \fK}(P)$. 
 \end{theorem}

Quantitative estimates for the relation between analytic rank and Schmidt rank provide a means for obtaining quantitative bounds for important problems in additive combinatorics number theory  and algebraic geometry. 

We conjecture the following quantitative relation between the Schmidt rank and analytic rank. 
\begin{conjecture}[d]\label{main}  For any $d\geq 2$ there exists $c_d>0$ such that  $r_{\mF _q}(P)\leq c_d a_{\mF _q}(P)$
for any  multilinear polynomial of degree $d$. 
\end{conjecture} 

\begin{remark}\label{leq}
\begin{enumerate}
 \item The inequality $a_{\mF _q}(P)\leq r_{\mF _q}(P) $ is known. See \cite{Ap, lovett-rank}.
 \item It is easy to see that Conjecture \ref{main} holds for $d=2$ with $c_2=1$.
 \item In \cite{Mi} it was shown that there exist  constants $c_d, e_d$ such that for any  multilinear polynomial of degree $d$ we have $r_{\mF _q}(P)\leq c_d (a_{\mF _q}(P))^{e_d}$ (earlier work \cite{bl} gave ineffective bounds). Conjecture \ref{main} states that one can take $e_d=1$. 
 \end{enumerate}
\end{remark}

In his paper \cite{S}, Schmidt proved the following result

\begin{theorem}[Schmidt]\label{Schmidt1}For any $d\geq 2$ there exists 
$D _d$  such that for any complex multilinear  polynomial
 $P: V_1\times \dots \times V_d \to \mC$ we have $r_\mC (P)\leq D_dg$ where $g:=\mathrm{codim}_{\bV_2\times \dots \times \bV_d } \bZ _P $. 
\end{theorem}

The main goal of this paper is the proof of the following result.

\begin{theorem}\label{maint} 
\begin{enumerate}
\item 
Assuming the validity Theorem \ref{Schmidt1}
 for multilinear polynomials of degree $d$ over algebraically closed fields of finite characteristic 
we show that 
Conjecture \ref{d}(d) implies  the validity of Conjecture \ref{main}(d) for $\fK = \mF _q, q> d$ with 
$c_d= \frac{D_d\kk _d}{1-\log _q(d)}$.
\item Theorem \ref{Schmidt1}
holds for multilinear polynomials of degree $3$ over algebraically closed fields of characteristic $\geq 3$ with $D_3=2$.
\end{enumerate}
\end{theorem}
The following statement follows now from Theorem \ref{33}(2).
\begin{corollary}   Conjecture \ref{main}(3) holds with  $c _3=3$.
\end{corollary}

The proof of Theorem \ref{maint} is based on an extension of the {\it rough bound} from \cite{hr}, and as such the method of proof provided for  Conjecture \ref{main}(3) is completely different that the approach taken in \cite{bl}, \cite{Mi}. 

The conjectured bound in \ref{main} can be viewed as a special case of the conjectured bounded for the inverse theorem for the Gowers norms over finite fields. Let $f:V \to \mathbb C$. The Gowers uniformity norms of $f$, introduced in the study of 
arithmetic progressions in subsets of the intergers, 
are defined as follows
\[
\|f\|^{2^d}_{U_d} =\frac{1}{|V|^{d+1}}\sum _{x, , h_1, \ldots, h_d \in V} \Delta_{h_d} \ldots \Delta_{h_1} f(x)
\]
where  $\Delta_{h} f(x) = f(x+h) \overline{f(x)}$. 

\begin{conjecture}\label{icgn} Let $\fK=\mF_p$, and let $d>p$. There exists a constant $F=F(d, k)$ such that the following holds: for any $\delta>0$,  any  $\fK$-vector space $V$, any $f:V \to \mathbb C$ satisfying $\|f\|_{U_d} \ge \delta$, there exists a degree $d-1$ polynomial $P$ such that $|\frac{1}{|V|}\sum_{x \in V} f(x)\psi(P(x))| \ge \delta^{F}$. 
\end{conjecture}

When  $f= \psi(Q)$ where $Q$ is a polynomial of degree $d$, then $\Delta_{h_d} \ldots \Delta_{h_1} f(x) = \psi( \tilde Q(h_1, \ldots, h_d))$, where $\tilde Q$ is the multilinear symmetric form associated with $Q$. Via Fourier analysis  the conjectured bound in \ref{main} implies the validity of Conjecture $\ref{icgn}$ in this case. \\

Finally we present a  conjecture relating the rank of a subspace over a field $\fK$ to its rank over the algebraic closure, and prove it in a special case.

\begin{conjecture}\label{cd} There exists a constant $E_d$ such that $r_\fK (L)\leq E_d r_{\bar \fK }(L)$ for any field $\fK$ and a linear subspace $L\subset (V_1\otimes \ldots \otimes V_d)^\vee$.
\end{conjecture}

\begin{theorem}\label{2} Conjecture \ref{cd} holds for $d=2$, with $E_d=2$. 
\end{theorem}


\section{Proof of Theorem \ref{2}}

In this section we prove Theorem \ref{2}. 
We start the proof of  with
the following result.

\begin{proposition}\label{22} For any field $\fK$ and a linear subspace  $L\subset \mathrm{Hom}_\fK (W,V)$ there exists $\fK$-subspaces 
$W'\subset W$ and  $V'\subset V$ such that $\mathrm{dim}(W/W')=\mathrm{dim}(V')=\ti r_{\fK}(L)$ and $l(W')\subset V'$ for all $l\in L$ where  $\ti r(L):= \max _{B\in L}r(B) $.
\end{proposition}

\begin{proof}  
We start with the follows result of Kronecker (see \cite{Kr}).

\begin{definition} 
\begin{enumerate}
\item Let $V,W$ be $\fK$-vector spaces. A pair $A,B\in \Hom (V,W)$ is irreducible is there is no 
non-trivial decomposition $V= V_1\oplus V_2$, $W = W_1\oplus W_2 $ such $A(V_i), B(V_i)\subset W_i$, for $i=1,2$.
\item For any $n\geq 0$ we denote by $A_n,B_n : \fK ^n\to \fK ^{n+1}$ maps such that $A_n(e_i)=f_i$, $B_n(e_i)= f_{i+1}$ for $1\leq i\leq n$, and denote by $A'_n,B'_n : \fK ^ {n+1}\to \fK ^n$ maps such that $A'_n(e_i)=f_i$  for $1\leq i\leq n$, $A_n(e_{n+1})=0$ and $B'_n(e_1)=0, B'_n(e_ {i+1})= f_i$ for  $2\leq i\leq n+1$. 
\end{enumerate}
\end{definition}

\begin{claim} Let  $ A, B\in \Hom (V,W)$ be an  irreducible pair of linear maps between are finite dimensional $\fK$-vector spaces. Then there
exist automorphisms 
$T\in \Aut(V),S\in \Aut(W)$ such that either 
\begin{enumerate}
\item $A$ is an isomorphism or 
\item $B$ is an isomorphism or 
\item there exist automorphisms 
$T\in \Aut(V)$, $S\in \Aut(W)$ such that $(SAT, SBT)= (A_n,B_n), n=\dim(V)$ such that  $(SAT, SBT)= (A_n,B_n)$ or
\item there exist automorphisms 
$T\in \Aut(V)$, $S\in \Aut(W)$ such that $(SAT, SBT)= (A'_n,B'_n) , n=\dim(W)$.
\end{enumerate}
\end{claim}

\begin{lemma}\label{Kr}If $|\fK| =\infty$ and $A,B\in \mathrm{Hom} _\fK(V,W)$  linear maps such that
 $r(A)\geq r(A+tB)$ for all $t\in \fK$. Then
 $B(\mathrm{Ker}(A))\subset \mathrm{Im}(A)$.
\end{lemma}

\begin{remark}The assumption that $|\fK| =\infty$ is  necessary. Indeed in the case when   $\fK =\mF _q$ we can  take $V=W=\fK [\mF _q]$ with the basis $\{e_x\}_{x\in \mF _q}$,  $A(e_x)=xe_x$ and $ B=Id$. Then  $r(A+tB) =r(A)=q-1$ for all $ t\in \fK$ but  $B(\mathrm{Ker}(A))\not \subset \mathrm{Im}(A)$.
\end{remark}
\begin{proof} Consider first the case when a pair $(A,B)$ is irreducible. If $\mathrm{Im}(A)=W$ or $\mathrm{Ker}(A)=\{0\}$ we obviously have the inclusion  $B(\mathrm{Ker}(A))\subset \mathrm{Im}(A)$. In cases $(1)$ and $(4)$ the map $A$ is onto and in the case $(3)$ the map $A$ is an imbedding. In the case $(2)$ we have $\dim(V)=\dim(W)$ and either $A$ is onto or $r(A)<r(B)$. We claim that the assumption that $r(A)<r(B)$ leads to a contradiction.

Let $S \subset \bA$ be the subset of $s_0$ such that $r(s_0A+B) < \max _{s\in \bar \fK} r(sA+B) $. Then  $S (\fK)\subset \fK$ is finite. So  $r(sA+B) \geq r(B)$ outside a finite set of $s$,  and since   $|\fK| =\infty$ 
and there exists $s \in \fK$ such that $r(A+s^{-1}B)  \ge r(B)> r(A)$, which is a contradiction.  

\sms

Now consider the general case when a pair $(A,B)$ is a finite direct sum of irreducible pairs $(A_i, B_i), i\in I$. Since, as we have seen in the previous paragraph,  the condition $r(A_i)\leq r(A_i+tB_i)$ is automatically true outside a finite set of $t\in \fK$ and  $|\fK| =\infty$ the assumption that  $r(A)\geq r(A+tB)$ for all $t\in \fK$ implies that that  $r(A_i)\geq r(A_i+tB_i)$ outside a finite set of $t\in \fK ,i\in I$. So  $B_i(\mathrm{Ker}(A_i))\subset \mathrm{Im}(A_i), i\in I$. Therefore  $B(\mathrm{Ker}(A))\subset \mathrm{Im}(A)$.

\end{proof}

\begin{lemma}\label{inf} Proposition  \ref{22} holds in the case when $|\fK| =\infty$. 

\end{lemma}

\begin{proof} 
Choose $A\in L$ such that $r(A)= \ti r(L)$ and define $W':= \mathrm{Ker}(A), V':= \mathrm{Im}(A)$.   As follows from Lemma \ref{Kr} we have $l(W')\subset V'$ for all $l\in L$. But $\mathrm{dim}(W/W')=\mathrm{dim}(V')=r(A) $. \end{proof}

Now consider the case when the field  $\fK$ is finite.
Let $ \mathrm{\bG r}$ be the Grassmanian of subspaces $W'\subset W$ of codimension $\ti r(L)$, and let $ \mathrm{\bG r'}$ be the Grassmanian of subspaces $V'\subset V$ of dimension $\ti r(L)$  and 
 $\bX _{\ti r_K(L)}\subset  \mathrm{\bG r}\times  \mathrm{\bG r'}$ be the subvariety of pairs
$(W',V')$ such that
and $l(W')\subset V'$ for all $l\in L$. It is clear that the subvariety  $\bX \subset  \mathrm{\bG r}\times \bG r'$ is closed. Therefore it is proper.

Let $K=\fK (t)$.
Since the field $K$ is infinite it follows from Claim \ref{inf} that $\bX (K)\neq \emp$. So  there exists a rational $\fK$-morphism $\hat  f: \bP ^1\to \bX$. Any such morphism is regular outside a finite $\fK$- subset $S\subset \bP ^1$. So we obtain a regular $\fK$-morphism $\ti f: \bP ^1 \setminus S \to \bX$. Since 
$\bX$ is proper,  $\ti f$ extends to a regular $\fK$- morphism $f: \bP ^1 \to \bX$. Let $(W'_0,V'_0):= f(0)\in \bX (\fK)$. By definition of the variety $\bX$ we see that $l(W'_0)\subset V'_0$ for all $l\in L$.
\end{proof}

Now we can finish the proof of Theorem \ref{2}.

As follows from Lemma \ref{2} there exists $\fK$-subspaces $V'\subset V$, $W' \subset W$ such that $l(W'_0)\subset V'_0$ for all $l\in L\otimes _\fK K$ and
 $\mathrm{dim}(W/ W')+  \mathrm{dim}(V')= 2 r_{\bar \fK}(L) $. So $ r_\fK (L) \leq 2 r_{\bar \fK}(L)$.

\section{Rough bound}

\subsection{A lemma on the codimension}
Let $K$ be an infinite field,  let $A$ be a finitely generated $K$-algebra of (Krull) dimension $n$ and  let $L \subset A$ be linear subspace of $A$. We denote by   $ J \subset A$ the ideal generated by $L$ and denote by 
 $ A_L$ the quotient algebra $A/J$.

\begin{lemma}\label{al}
If $\mathrm{dim}A_L < n$ then there exists a finite collection of subspaces  $L_i \subsetneq L$, $i \in I$ such that 
 the algebra $A_l$ has dimension $<n$ for any $l\in L\sm \bigcup _{i\in I}L_i$.
 \end{lemma}

 \begin{proof}

Let $X_i, \  i = 1,\ldots,k$,  be irreducible components of $SpecA$ of dimension $n$.

For every $i$ we denote by $K_i$ the field of rational functions on the component $X_i$, 
and consider the natural morphism $\nu_i: A \to K_i$. 

\sms

Since $\mathrm{dim}A_L < n$ the image 
  $\nu_i(L)$ is not zero. Hence the subspace $ L_i := L \bigcap \mathrm{Ker}(\nu_i) \subset L$ is strictly contained in $L$.

\sms

If an element $l \in L$ does not belong to any of the spaces $L_i$ then its image in every field $K_i$ is
not zero. This implies that the dimension of the algebra $A_l = A / Al$ is less than $n$.
\end{proof}
\begin{corollary}\label{m}There exist elements  $l_1,\dots ,l_m\in L, m:= \mathrm{dim}(A)-\mathrm{dim}(A/L)$  such that $\mathrm{dim}(A/ J)=\mathrm{dim}(A/L) $ where $ J\subset A$ is the ideal generated by  $l_1,\dots ,l_m$.
\end{corollary} 
\begin{proof}Induction in m.
\end{proof}


\subsection{A proof of the rough bound}

In this subsection we present a proof of a generalization of the {\it rough bound} from \cite{hr}.

For any subset $\Theta \subset  \mF _q[x_1,\dots ,x_n] $ we denote by $\bX _ \Theta \subset \bA ^n$ the subscheme which is the intersection of zeros of $\theta \in  \Theta$.

\begin{proposition}\label{rb} Let $M\subset \mF _q[x_1,\dots ,x_n]$ be a linear subspace of polynomials of degrees $d$ such that
$\bY :=\bX _{M}$ is of dimension $m$. Then
 $|\bY ( \mF _q)|\leq q^md^c, c:=n-m$.
\end{proposition}
\begin{proof} Let $F$  be the algebraic closure of $\mF _q$.

As follows from Lemma \ref{al} there exists $P_i\in M\otimes _{\mF _q}F,1\leq i\leq c$ such 
that $\mathrm{dim}(\bY ')=m$ where $\bY ' = \bX _\mcP$ , $\mcP :=\{ P_i\} _{1\leq i\leq c} $.

\ms

It is clear that  $\bY (\mF _q)$ is the  intersection of $\bY$ with hypersurfaces $S_j, 1\leq j\leq n$ defined by the equations $h_j( x_1,\dots ,x_n)=0$ where 
$h_j( x_1,\dots ,x_n)= x_j ^q-x_j$. Since $\bY \subset \bY '$ we see that  $\bY (\mF _q)$ is contained in the   intersection of $\bY '$ with hypersurfaces $S_j, 1\leq j\leq n$.

For $j=1, \ldots, m$ let  $H_j=\sum _{i=1}^n a_j^ih_i$, $a_j^i \in F'$ be a  linear combination of the $h_j$ such that coefficients $a_j^i $ are algebraically independent over $F$ where $F'/F$ is a transcendental extension. We denote by 
$\bZ _1,\ldots,\bZ _m\subset \bA ^n$ be the corresponding hypersurfaces and define  $\bB _j:= \bY '\cap (\bigcap _{i=1}^j\bZ _i)$. 
\begin{claim}\label{j} Each component $\bC$ of $\bB _j$ is of dimension $m-j$.
\end{claim}
\begin{proof} The proof is by induction in $j$. The statement obviously is true for $j=0$. 

Any component $\bC$ of $\bB _{j+1}$ is a component of an intersection 
$\bC '\cap \bZ _{j+1} $ for some component $\bC '$ of  $\bB _j$. By induction $\mathrm{dim}( \bC ')=m-j$. So  not all the functions $h_j$ vanish on $\bC '$. Hence by the genericity of the choice of linear combinations $\{H_j\}$ we see that 
$H_{j+1}$ does not vanish on $\bC '$ and therefore $  \bC '\cap \bZ _{j+1} $ is 
of pure dimension $m-j-1$.  
\end{proof}

As follows from Claim \ref{j} the intersection
 $\bY '\cap  \bZ _1 \cap \dots \cap  \bZ  _m$ has dimension $0$.  Therefore the  B\'ezout's theorem implies that  $|\bY '\cap  \bZ _1 \cap \dots \bZ _{n-c}|\leq q^md^c$. Since $\bY (\mF _q)=\bY  \cap  \bZ _1 \cap \dots \cap \bZ _n\subset \bX \cap  \bY _1 \cap \dots  \cap \bY _{n-c} $ we see that 
 $|\bY (\mF _q)|
\leq q^m d^c$.
\end{proof}


\section{Proof of Theorem \ref{maint}(1)}
\begin{proof} Let $\fK = \mF _q ,
P: V_1\times \dots \times V_d \to \fK$ 
be a multilinear polynomial and $g= \mathrm{codim}_{\bV _2\times \dots \times \bV _d }  \bZ _P $ and $\bar \fK$ be the algebraic closure of $\mF _q$.

\sms

Since ( see Lemma \ref{ind}) 
$a _\fK(P)= \sum_{i=2}^d \dim (V_i)- \log _q( |\bZ _P(\fK)|)$ it follows from Proposition \ref{rb} that  
$|\bZ _P(\fK)|\leq d^gq^{n-g}$ where $n= \sum_{i=2}^d \dim (V_i )$. So $\log _q( |\bZ _P(\fK)|) \leq g \log _q(d)+n-g$. We see that
$a_\fK(P)\geq g(1-\log _q(d))$. 

Assuming the validity of 
Theorem \ref{Schmidt1}  for multilinear polynomials of degree $d$ over $\bar \fK$  we see that $g
\geq  \frac {r_{\fK}(P)}{ \kk _d  D _d}$ and therefore 
$a(P)\geq  \frac {r_{\fK}(P)}{ \kk _d  D _d}(1-\log _q(d))$.  Theorem \ref{maint}(1) is proven.
\end{proof}
In the next section we prove the second part of Theorem \ref{maint}.


\section{The adaptation of the Schmidt's result for  fields of finite characteristic in the case when $d=3$.}

We use notations from Definitions \ref{11},\ref{d1}. In \cite{S} W. Schmidt proved the following result.
\begin{theorem}[Schmidt]\label{Schmidt}For any $d\geq 2$ there exists 
$D _d$  such that for any complex multilinear  polynomial
 $P: V_1\times \dots \times V_d \to \mC$ we have $r_\mC (P)\leq D_dg$ where $g:=\mathrm{codim}_{\bV_2\times \dots \times \bV_d } \bZ _P $
\end{theorem}
His proof extends to any algebraically closed field, and we sketch it here for the three-dimensional case.

\ms

In this section we fix an algebraically closed field $\fK$ and write $r(P)$ instead of $r_\fK (P)$. Let $\bX$ be an algebraic $\fK$-variety.
Since our field $\fK$ is algebraically closed any  constructible subset 
 $Y\subset X$ defines uniquely a subset $\bY \subset \bX$ such that $Y=\bY (\fK)$. We say that a constructible subset  $Y\subset X$ is open if the subset $\bY \subset \bX$ is open and we  define $\dim(Y):= \dim(\bY)$.

\ms

For a trilinear polynomial  $P: U\times  V\times W \to \fK$ we denote by $Z_P\subset V\times W$  the 
constructible subset of points $(v,w)$ such that $P(u,v,w)=0$ for all $u\in U$.

The main goal of this section is to prove the following theorem:
\begin{theorem} \label{Schmidt(f)}
For any trilinear polynomial $P: U \times  V\times W \to \fK $
we have $r(P)\leq 2g$ where $g:=\mathrm{codim}_{V\times W} Z _P $.
\end{theorem}
Before proving Theorem \ref{Schmidt(f)} we remind the reader of some results from Algebra.
\subsection{ Linear Algebra}

\begin{definition}For a  bilinear form $B: U\times V\to \fK$ we write
\begin{enumerate}
\item 
$S_U(B):=\{u\in U| B(u,v)=0, \forall v\in V\}$.
 \item $S_V(B):=\{v\in V| B(u,v)=0, \forall u\in U\}$. 
\end{enumerate}
\end{definition}
\begin{remark}$r(B)= \codim_US_U(B) = \codim_VS_ V(B)$.
\end{remark}

Let $X$ be a smooth curve over $\fK$ , let $t_0\in X$ and let $B(t): U\times V\to \fK , t\in X $ be a family of bilinear forms. We write $B:=B(t_0)$, $S_U:= S_U(B)$ , $S_V:= S_V(B)$, and $C:=(\partial B(t)/ \partial t) _{t=t_0}$.

\begin{claim}\label{0} If $r(B(t))\leq r(B), t\in X$ then 
$C_{|S_U\times S_V} \equiv 0 $.
\end{claim}
\begin{proof} We show that the assumption that $C_{|S_U\times S_V} \not \equiv 0 $ leads to a contradiction.
To simplify notations we assume that $\bX =\bA$ and $t_0=0$.

To start with we choose bases $e_i,f_j$, where $1\leq i\leq \dim(U), 1\leq j\leq \dim(V)$ of $U$ and $V$ such that 
\begin{enumerate}
\item $B(e_i,f_j)= \delta _{i,j}, 1\leq i,j\leq r(B)$. 
\item $B(e_i,f_j) =0$ if either $i>r(B)$ or $j>r(B)$.  

\end{enumerate}

Then $S_U$ is the span of $\{e_i\}$ , where $r(B)+1\leq i\leq \dim(U) $ and $S_V$ is the span of $\{f_j\}$, where $r(B)+1\leq j\leq \dim(V) $. 
Since   $C_{|S_U\times S_V} \not \equiv 0 $ there exist a pair $(i_0,j_0)$  such that $i_0,j_0>r(B) $ and $C(e_{i_0}, f_{j_0})\neq 0$. 

Let $U'\subset U$, $\dim (U')=r(B)+1$ be the span of $\{e_i\}_{1\leq i\leq r(B)} $ and of $e_ {i_0} $ and $V'\subset V$, $\dim (V')=r(B)+1$ be the span of $\{f_j\},_{1\leq j\leq r(B)} $ and of $f_ {j_0} $. 

We denote by  $\D (t)$ be the determinant of the 
bilinear form  of $B(t)_{U'\times V '}$ with  respect to chosen bases of $U'$ and $V'$. The condition $r(B(t))\leq r(B),t\in \fK$ implies that $\D (t)\equiv 0$. On the other hand 
$\D (t)\equiv t C(e_{i_0}, f_{j_0}) \mod (t^2)$. 

This contradiction proves Lemma \ref{0}.

\end{proof}

\subsection{Transcendence degree}
\begin{definition}
\begin{enumerate}
\item For a field extension $K/\fK$ we define the transcendence degree $tr (K/\fK)$ to be  the minimal  $n\geq 0$ such that there is no imbeddings of  $\fK (u_0,u_1,\dots ,u_n)\ho K$.
\item For a point $w \in K^N$ we denote by $\fK (w)\subset K$ the subfield generated by $\fK$ and the coordinates of $w$  and write $tr _\fK (w):= tr (\fK (w)/\fK) $.
\item Let $\bX \subset \bA ^N$ be an irreducible $\fK$-subvariety.  A point 
$w\in \bX (K)$ is generic if it is not contained in any proper  $\fK$-subvariety of $X$.
\end{enumerate}
\end{definition}
\begin{claim}\label{tr} 
\begin{enumerate} 
\item Let $\bX \subset \bA ^M$ be an irreducible $\fK$ variety and $w \in \bX (K) $ be generic point. Then $tr (\fK(w)/\fK)=\dim(X)$.
\item Let $p:  \bA ^ M\to \bA ^ N$ be a $\fK$-linear map, $\bZ \subset \bA ^M$ an irreducible variety, and $\bY =p(\bZ)$. Let $v\in \bZ (K) $ be generic point and $w:=p(v) \in \bY (K) $.
Then $w\in \bY (K) $ is a generic point,  $\fK(w) \subset \fK(v) $ and 
 $tr (\fK(v)/\fK(w) )= \dim(\bZ) -\dim(\bY)$.
\end{enumerate}
\end{claim}

\subsection{Proof of Theorem \ref{Schmidt(f)}}
\begin{proof} Let  $P: U \times  V\times W \to \fK $ be as in Theorem 
\ref{Schmidt(f)} and $Z\subset Z_P$ be an irreducible component of the maximal 
dimension, $g:= \codim _{V\times W}Z$. We have to show that $r(P)\leq 2g$.

 Let $z_0=(v_0,w_0)$ be a generic point of $Z$, and let $Y\subset W$ be  the projection  of $Z$ on $W$. For $w\in W$ we denote by $P_w$ the  bilinear forms on $V$ given by
 $P_w(u,v): = P(u,v,w)$. 

Let  $S_U=\{u\in U| P_{w_0}(u,v)=0, \forall v\in V\}$,
 and let $S_V=\{v\in V| P_{w_0}(u,v)=0, \forall u\in U\}$. Since $w_0$ is a generic point of $Y$, we see that $r(P_y)\leq r(P_{w_0}) $ for all $y\in Y$.
Since $w_0\in Y$ is a generic point,  $Y$ is smooth at $w_0$ and we can define $S_W:= T_Y(w_0)$ the tangent space at $w_0$.. 

 It follows from 
 Calim \ref{3} that the following statement implies the validity of Theorem \ref{Schmidt(f)}.

\begin{proposition}\label{Z}
\begin{enumerate}
 \item $P_{| S_U\times S_V\times S_W}\equiv 0$.
\item $\codim_U(S_U)+\codim_W(S_W)+\codim_V(S_V)\leq 2g$. 
\end{enumerate}
\end{proposition}
\begin{proof}  
To show that  $P_{| S_U\times S_V\times S_W}\equiv 0$ we have to show that for 
any $C\in S_W$ the restriction of the bilinear form $P_C$ on $ S_U\times S_V $ is 
equal to $0$. So we fix  $C\in S_W$. Since $P_w$  for $w\in W$ is a linear function on 
$W$ we have  $P_C= (\partial P_w/\partial C) _{w=w_0}$. So we have to show that 
$(\partial P_w/\partial C) _{ w=w_0} |_{S_U\times S_V} \equiv 0 $.

Choose a smooth curve on $Y$ passing through $w_0$ and tangent to $C$. In other words choose a map $\phi  : X\to Y$ of a smooth curve 
$X$ to $Y$ and a point $t_0\in X$ such 
such that $\phi  (t_0)=w_0$ and $C=(\partial \phi(t)/ \partial t) _{t=t_0}$. Since $r(P_y)\leq r(P_{w_0)}$ for $y\in Y$ the family $B(t):= P_{\phi (t)}$ of bilinear forms on 
$U\times V$ satisfies the assumption of  Claim \ref{0}.  Therefore  $(\partial P_w/\partial C)_{ w=w_0}| _{S_U\times S_V} \equiv 0 $, and thus 
$P_C |_{S_U\times S_V} \equiv 0$.

\sms

Since $\codim _U(S_U)=\codim _V(S_V)= r(P_{w_0})$ it follows from Claim \ref{tr} that $\codim_U(S_U)+\codim_W(S_W)= \codim_V(S_V)+\codim_W(S_W)\leq g$. So 
$\codim_U(S_U)+\codim_W(S_W)+\codim_V(S_V)\leq 2g$. 
\end{proof}
\end{proof}


\begin{thebibliography}{99}

\bibitem{ah}  Ananyan, T., Hochster, M.  {\em Small subalgebras of polynomial rings and Stillman's Conjecture}. J. Amer. Math. Soc. 33 (2020), 291-309. 
\bibitem{bl} Bhowmick A., Lovett S. {\em Bias vs structure of polynomials in large fields, and applications in effective algebraic geometry and coding theory}. ArXiv:1506.02047.
\bibitem{clp} Croot, E.,  Lev, V.P., Pach, P. {\em Progression-free sets in $\mZ_4^n$ are exponentially small}. Ann. of Math., vol. 185, 331-337, 2017.
\bibitem{D} Dersken. H. {\em The $G$-stable rank for tensors} {arxiv.org/abs/2002.08435}.
\bibitem{hr} Hrushovski, E.  {\em The Elementary Theory of the Frobenius Automorphisms} arXiv: 0406514.
\bibitem{eg} Ellenberg, J, Gijswijt, D. {\em On large subsets of $\mF_q^n$ with no three-term arithmetic progression}. Ann. of Math., vol. 185, 339-343, 2017.
\bibitem{janzer} Janzer, O. {\em Polynomial bound for the partition rank vs the analytic rank of tensors} Discrete Analysis 2020:7. 
\bibitem{Ap} Kazhdan, D. and Ziegler, T.{\em  Approximate cohomology}. Selecta Math. (N.S.) 24 (2018), no. 1, 499 - 509.
\bibitem{Kr} Kronecker, L. {\em Algebraische Reduction der Scharen bilinearer Formen}, Sitzungsber. Akad. Berlin, Jbuch. 22 (169)
(1890) 1225-1237.
\bibitem{gw} Gowers, T, Wolf, J. {\em Linear Forms and Higher-Degree Uniformity for Functions On $\mF_p^n$}. Geom. Funct. Anal.  21, 36-69 (2011).
\bibitem{lovett-rank} Lovett, S. {\em The Analytic rank of tensors and its applications}. Discrete Analysis 2019:7 
\bibitem{Mc} F. S. Macaulay, \emph{The algebraic theory of modular systems}, Cambridge University, 1916.
\bibitem{Mi} Mili\'cevi\'c, L.  {\em Polynomial bound for partition rank in terms of analytic rank} Geom. Funct. Anal. 29 (2019), no. 5, 1503-1530. 
\bibitem{S} Schmidt, W.  M. {\em The density of integer points on homogeneous varieties}. Acta Math. 154 (1985), no. 3-4, 243-296. 
\bibitem{tt} Tao, T. {\em A symmetric formulation of the Croot-Lev-Pach-Ellenberg-Gijswijt capset bound}. terrytao.wordpress.com/2016/05/18/a-symmetric-formulation-of-the-croot-lev-pach-ellenberg-gijswijt-capset-bound/



\end{thebibliography}
\end{document}